\documentclass[final,3p]{elsarticle}

\usepackage{lineno,hyperref}
\modulolinenumbers[5]

\usepackage{amssymb}
\usepackage{amsmath}
\usepackage{mathtools}
\usepackage{amsthm}
\usepackage{color}

\newtheorem{thm}{Theorem}
\newtheorem{pro}{Proposition}

\newdefinition{rmk}{Remark}
\newproof{pf}{Proof}

\newcommand{\g}{\mathfrak{g}}
\newcommand{\J}{\mathcal{J}}
\newcommand{\oen}{\mathfrak{o}(n)}
\newcommand{\otrzy}{\mathfrak{o}(3)}
\newcommand{\eseldwa}{\mathfrak{sl}(2,\mathbb{R})}

\begin{document}

\begin{frontmatter}

\title{Six-dimensional product Lie algebras admitting integrable complex structures}

\author[mymainaddress]{Andrzej Czarnecki}
\cortext[mycorrespondingauthor]{Corresponding author}
\ead{andrzejczarnecki01@gmail.com}

\author[mymainaddress]{Marcin Sroka}
 
\ead{marcin.sroka@student.uj.edu.pl}

\address[mymainaddress]{Jagiellonian University, \L ojasiewicza 6, 30-348 Krakow, Poland}

\begin{abstract}
We classify the 6-dimensional Lie algebras of the form $\g\times\g$ that admit an integrable complex structure. We also endow a Lie algebra of the kind $\oen\times\oen$ ($n\geq 2$) with such a complex structure. The motivation comes from geometric structures \`a~la~Sasaki on $\g$-manifolds.
\end{abstract}


\end{frontmatter}

\section{Introduction.}

A complex structure on a Lie algebra $\mathfrak{h}$ is an endomorphism $\J: \mathfrak{h} \rightarrow \mathfrak{h}$ such that $\J^2= -id$. It corresponds to a left invariant almost complex structure on any Lie group $H$ with $T_eH=\mathfrak{h}$. We say that a complex structure $\J$ on $\mathfrak{h}$ is integrable if
\begin{align*}
N(v,w):=[v,w]+\J[\J v,w]+\J[v,\J w]-[\J v,\J w]=0
\end{align*}
for any $v,w \in \mathfrak{h}$. By the Newlander-Nirenberg theorem, via a left invariant trivialisation of the tangent bundle of $H$, this condition is likewise equivalent to the integrability of the corresponding left invariant almost complex structure on $H$.

Classification of integrable complex structures on real Lie algebras is a well established problem, cf. a summary of results on their existence in \cite{classif}. In dimension 6, the question is settled only for special -- abelian -- complex structures, cf. \cite{andr}, and for nilpotent algebras, cf. \cite{nilmanifolds, sal}. The present paper focuses on a different class of 6-dimensional Lie algebras that split as a product $\g \times \g$ for a 3-dimensional Lie algebra $\g$. We identify all such algebras admitting integrable complex structures. This problem was studied in the special cases of $\otrzy \times \otrzy$ and $\eseldwa \times \eseldwa$ by Magnin in \cite{mag1, mag2}, where he also classified all possible integrable complex structures. Such complete description will not be our concern here, but we note that the classification given in \cite{nilmanifolds} covers types (1), (2), and (3) of our Proposition \ref{alg}.

Our main motivation comes from differential geometry. As explained in detail in \cite{MAR}, the existence of complex structures on 6-dimensional product algebras has immediate applications to the recently developed theory of non-abelian, higher-dimensional structures \`a la Sasaki. To get an idea of the problem, consider a 3-dimensional Lie group $G$ acting freely on an odd-dimensional manifold $M$. Suppose that this action preserves some transverse complex structure -- a complex structure on the sub-bundle $\nu$ transverse to the orbits. We cannot hope to extend this complex structure to the whole tangent bundle -- since the dimension is odd -- but an interesting problem is to extend it to the $T\left(M\times G\right)$. Since the tangent space at a point $x$ in that product splits (as a vector space) $T_x\left(M\times G\right)=\nu_x\times\g\times \g$, this rises -- and reduces to -- the question of finding an integrable complex structure on $\g\times \g$. Then the transverse complex structure on $M$ can be studied in terms of an ordinary complex structure on $M\times G$. Recall that a manifold $S$ is Sasakian if its Riemannian cone $S\times \mathbb{R}$ is K\"ahler -- and thus the above approach is a starting point for natural generalisations.

We point out that an integrable complex structure on its Lie algebra does not turn a Lie group into a complex Lie group. In fact, every compact Lie group of even dimension admits an integrable left invariant complex structure, cf. \cite{1, 2}, while it is well-known that only tori can be compact complex Lie groups.

In the last section we provide an explicit integrable complex structure for every algebra of the type $\oen\times \oen$.

\section{Complex structures on 6-dimensional product algebras.}

Recall that the 3-dimensional Lie algebras were classified into 9 types by Bianchi. We use a variant, a classification from \cite{bia} by the dimension of the derived algebra and Jordan decomposition of certain automorphism acting upon it. We include the statement for convenience.

\begin{pro}{\cite{bia}}\label{alg} Let $e_1$, $e_2$ and $e_3$ be a basis of $\mathbb{R}^3$. Up to isomorphism of Lie algebras, the following list yields all Lie brackets on $\mathbb{R}^3$
{\renewcommand\labelenumi{(\theenumi)}
\begin{enumerate}
\item $[e_1,e_3]= 0$, $[e_2,e_3]= 0$, $[e_1,e_2]= 0$ 
\item $[e_1,e_3]= 0$, $[e_2,e_3]= 0$, $[e_1,e_2]= e_1$
\item $[e_1,e_3]= 0$, $[e_2,e_3]= 0$, $[e_1,e_2]= e_3$ 
\item $[e_1,e_3]= e_1$, $[e_2,e_3]= \theta e_2$, $[e_1,e_2]= 0$ for $\theta\neq 0$ (the case $\theta = 1 $ is considered to be Bianchi's ninth type)
\item $[e_1,e_3]= e_1$, $[e_2,e_3]= e_1+e_2$, $[e_1,e_2]= 0$ 
\item $[e_1,e_3]= \theta e_1-e_2$, $[e_2,e_3]= e_1+\theta e_2$, $[e_1,e_2]= 0$  for $\theta \neq 0$
\item $[e_1,e_3]= e_2$, $[e_2,e_3]= e_1$, $[e_1,e_2]=e_3$ 
\item $[e_1,e_3]=-e_2 $, $[e_2,e_3]=e_1 $, $[e_1,e_2]= e_3$ 
\end{enumerate}}
\end{pro}

We fix some notation. Whenever we write $(x,y,z)\in\g$, it is understood in the appropriate basis above. If any other basis $\{u,v,w\}$ is used, we write $Xu+Yv+Zw$.

The direct product $\g\times\g$ inherits the bracket operation on each factor from $\g$: $[(u,v),(t,s)]=([u,t],[v,s])$. We keep the distinction between two copies of $\g$ inside $\g\times\g$ by adding asterisks to the second copy. Any vector decorated with an asterisk is understood to lie in $\g^*=0\times \g$, while those without it lie in $\g=\g\times 0$. We tacitly use the natural isomorphism between two copies, $(v,0)^*=(0,v)$. We also distinguish the two components of a complex structure $\J$ -- it will be convenient to work with $J$ and $J^*$ as in $\J v=(Jv,J^*v)$ to indicate its $\g$- and $\g^*$-parts separately. As a rule, virtually every vector on which we act in the lengthy proofs lies in $\g$.

To finish the preliminaries we note the following to use frequently in what follows.

\begin{rmk}\label{redu}
Let $\mathfrak{h}$ be a Lie algebra with a complex structure $\J$. One can easily check that for any $v,w \in \mathfrak{h}$
\begin{align*}
N(v,w)= - N(\J v,\J w)= \J N(\J v,w)= \J N(v,\J w)
\end{align*}
\end{rmk}




\begin{thm}\label{main}
A real 6-dimensional Lie algebra of the form $\g \times \g$ (for some real 3-dimensional Lie algebra $\g$) carries an integrable complex structure if $\g$ is of type (1), (2), (3), (6), (7), (8), and (4) with parameter $\theta=1$ in Proposition \ref{alg} above. There is however no such structure for the type (5), and for all other parameters in (4).

\end{thm}

\begin{proof}[Proof of the existence]
We consider two cases.
\begin{enumerate}
\item There is $u\in\g$ such that the adjoint $[u,\cdot]$ has a complex eigenvalue $A+Bi$ with $B \neq 0$. Then it has a Jordan basis $\{u, v, w \}$ with $[u,v]  =   Av+Bw$, $[u,w] = -Bv+Aw$ and there is a complex structure $\J$ given by
\begin{align*}
\J u &=  u^*  \\
\J v &=  w \\
\J v^* &=  w^*
\end{align*}
which of course suffices to define it completely. Then
\begin{itemize}
\item $N(u,u^*)=N(v,w)=N(v^*,w^*)=0$ because it concerns only 2-dimensional subspaces invariant by $\J$;
\item $N(u,v)=0$ and $N(u^*,v^*)=0$ is sufficient for every other pair of vectors by Proposition \ref{redu}. We compute, by symmetry between $\g$ and $\g ^*$, only the first case.
\begin{align*}
& [u,v]+\J[\J u,v]+\J[u,\J v]-[\J u,\J v]= Av+Bw+0+\J[u,w]-0 \\
=&\quad Av+Bw+\J(-Bv+Aw)=Av+Bw-Bw-Av=0 
\end{align*}
\end{itemize}

This settles cases (6), (7), and (8) in Proposition~\ref{alg} -- the appropriate vectors $u, v, w$ are collected in Proposition~\ref{STR} far below.
\item There is a non-zero $u\in\g$ such that the adjoint $[u,\cdot]$ has a double (resp. triple, if it is 0) real eigenvalue $\alpha$ with two (resp. three) linearly independent eigenvectors $v, w$, that form together a Jordan basis $\{u,v,w\}$. Define $\J$ by
\begin{align*}
\J u & =  u^*  \\
\J v & =  w \\
\J v^* & =  w^*
\end{align*}
Then, as before, the only conditions to check are $N(u,v)=0$ and $N(u^*,v^*)=0$, readily computed to be zero since 
\begin{align*}
& [u,v]+\J[\J u,v]+\J[u,\J v]-[\J u,\J v]= \alpha v+0+\J[u,w]-0 \\
=&\quad \alpha v + \J(\alpha w)= \alpha v - \alpha v 
\end{align*}
This settles cases (1), (2), (3), and (4) with $\theta=1$ -- the appropriate vectors are again given in Proposition~\ref{STR}.
\end{enumerate}
\end{proof}

Before we proceed, we introduce one additional convention: unless stated otherwise, the $J^*$-part of $\J$ will be suppressed and so the following (many) Nijenhuis brackets will often be understood as expressions in $\g$ -- without loss of generality, because $N\equiv 0$ iff both its $\g$- and $\g^*$-parts are. We will also forgo writing "$=0$" in subsequent equations: every Nijenhuis bracket in sight (or $\g$-, or $\g^*$-part thereof) is understood to equal 0. We feel that this helps keep the exposition shorter without causing too much confusion.

\begin{proof}[Proof of the non-existence of integrable complex structures -- case (4) with $\theta \neq 1$]
We begin with close examination of the algebra in question. In the canonical basis, the adjoint endomorphism $[(x,y,z),\cdot]$ of a vector is given by
$$
\left[
\begin{array}{ccc}
 -z & 0 & x \\
 0 & -\theta z & \theta y \\ 
 0 & 0 & 0 
\end{array}
\right]
$$
Its characteristic polynomial is $-t(t+z)(t+\theta z)$. If $z$ is non-zero (which is the only interesting case as we will see shortly), then in a Jordan basis $\{u,v,w\}=\{(x,y,z), e_1, e_2 \}$ this adjoint takes the form
$$
\left[
\begin{array}{ccc}
 0 & 0 & 0 \\
 0 & -z & 0 \\ 
 0 & 0 & -\theta z 
\end{array}
\right]
$$
and the brackets are $[u,v]=-zv$, $[u,w]=-z\theta w$, $[v,w]=0$.

Suppose that there is an integrable complex structure $\J$ on $\g\times\g$. We will proceed as follows:
{\renewcommand\labelenumi{(\theenumi)}
\begin{enumerate}
\item There is a vector $u$ such that $\J u=Ju+J^*u=\lambda u + J^*u$ for some real $\lambda$.
\item There is no $V$, non-trivial $\J$-invariant subspace of $\g$.
\item Any vector $u$ as in point (1) must be generic -- adjoint $[u,\cdot]$ must have a non-zero (real) eigenvalue.
\end{enumerate}
}

Regarding (1) -- this is simple linear algebra: the characteristic polynomial of $J\,:\,\g\longrightarrow \g$ is of degree 3 and thus has a real root. We call a resulting $u$ a \emph{quasi-invariant} vector. We have of course a degree of freedom in the choice of such a vector.

Regarding (2) -- suppose to the contrary. Such a $V$ must be 2-dimensional. We consider the three cases:

\begin{enumerate}
\item $[V,V]=0$. Then $V=span\{e_1,e_2\}$. There is a quasi-invariant $u=(x,y,1)$ (note the last coordinate) and for an $a\in V$ we compute $N(u,a)$.
\begin{align*}
& [u,a]+J[Ju,a]+J[u,Ja]-[Ju,Ja] \\
=&\quad [u,a]+\lambda J[u,a] + J[u,Ja]-\lambda [u,Ja]  \\
=&\quad [e_3,a]+\lambda J[e_3,a] + J[e_3,Ja]-\lambda [e_3,Ja] 
\end{align*}
because the other terms vanish. For $a=e_1$ and $\J e_1=Je_1=Xe_1+Ye_2$ (which forces~$Je_2=\frac{-1-X^2}{Y}e_1-Xe_2$) we get
\begin{align*}
& [e_3,e_1]+\lambda J[e_3,e_1] + J[e_3,Je_1]-\lambda [e_3,Je_1]  \\
=&\quad -e_1-\lambda Je_1 + J[e_3,Xe_1+Ye_2]-\lambda [e_3,Xe_1+Ye_2]  \\
=&\quad -e_1-\lambda Xe_1-\lambda Ye_2 - X(Xe_1+Ye_2)-Y\theta(\frac{-1-X^2}{Y}e_1-Xe_2)+\lambda Xe_1+\lambda\theta Ye_2 &
\end{align*}
or
$$
\begin{cases}
-1-\lambda X - X^2 +\theta(1+X^2)+\lambda X = (1-\theta)(-1-X^2) =0 \\ 
-\lambda Y-XY+\theta XY+\lambda\theta Y=(1-\theta)(-XY-\lambda Y)=0
\end{cases}
$$
which delivers the contradiction -- remember that $\theta\neq 1$.
\item $[V,V]\neq 0$, $[V,V]\subset V$. This must be a 1-dimensional subspace. A non-zero vector $v'$ in this subspace is a non-zero eigenvector for any linearly independent vector $u' \in V$. Without loss of generality, $u'=\alpha e_1+\beta e_2 + e_3$ and $v'$ is either $e_1$ or $e_2$. The other vector (and an eigenvector to the non-zero eigenvalue $\kappa$, equal to 1 or $\theta$) does not lie in $V$ and completes a Jordan basis $\{u',v',w'\}$ for the adjoint of $u'$. There is a quasi-invariant $u=xu'+yv'+w'$ (note the last coefficient). We compute the full Nijenhuis bracket $N(u,v')$, with $\J v'=Jv'=Xu'+Yv'$
\begin{align*}
&[u,v']+\J[\J u,v']+\J[u,\J v']-[\J u,\J v'] = [u,v']+\lambda\J[u,v']+\J[u,Jv']-\lambda [u,Jv'] &\\
=&\quad [xu'+yv'+w',v']+\lambda \J[xu'+yv'+w',v']+\J[xu'+yv'+w',Xu'+Yv']-\\
&\qquad-\lambda [xu'+yv'+w',Xu'+Yv'] &\\
=&\quad [xu',v']+\lambda \J[xu',v'] + \J[xu',Yv']+\J[yv'+w',Xu']- \lambda [xu',Yv']- \\
&\qquad-\lambda [yv'+w',Xu'] &\\
=&\quad [xu',v']+\lambda J[xu',v'] + J[xu',Yv']+ J[yv',Xu'] + \J[w',Xu']- \lambda [xu',Yv']- \\
&\qquad -\lambda [yv'+w',Xu']
\end{align*}
We see that this expresses $-X\kappa \J w'$ as a vector in $\g$, which is possible only if it is zero ($\J w'$ must not lie in $\g$, since $\g$ would then be an invariant subspace). But neither $\kappa$ nor $X$ can be zero -- if $X$ was 0, $\J$ would have a invariant direction, $v'$ --  a contradiction. We will use this argument repeatedly.
\item $[V,V]\neq 0$, $[V,V]$ is not in $V$. $[V,V]$ is contained in $[\g,\g]=span\{e_1,e_2\}$ which intersect $V$ along a 1-dimensional subspace $\mathbb{R}\cdot u'$. Take $w' \in V$ linearly independent of $u'$, observe we may assume $w'=(\alpha,\beta,1)$. Put $v'=[u',w'] \in [V,V]$, then the adjoint of $u'$ has a Jordan basis $\{u',v',w'\}$. Observe that we have $[v',w']=-\theta u'+(1+\theta)v'$ and $[u',v']=0$. The quasi-invariant vector $u=xu'+v'+zw'$ (note the middle coefficient) and $u'$ satisfying $\J u'=Ju'=Xu'+Yw'$ give us the full $N(u,u')$
\begin{align*}
& [u,u']+\lambda \J[u,u'] + \J[u,\J u']-\lambda [u,\J u'] &\\
=&\quad [xu'+v'+zw',u']+\lambda \J[xu'+v'+zw',u']+\J[xu'+v'+zw',Xu'+Yw']-\\
&\qquad -\lambda[xu'+v'+zw',Xu'+Yw'] &\\
=&\quad [zw',u']+\lambda\J[zw',u']+\J[zw',Xu']+\J[xu'+v',Yw']-\lambda[xu'+v',Yw']-\lambda[zw',Xu'] &\\
=&\quad -zv'-\lambda z\J v'-zX\J v'+xY\J v'+(1+\theta)Y\J v'-\theta Y\J u'-xY\lambda v'-\\
&\qquad -\lambda  Y(1+\theta)v'+\lambda\theta  Yu'+zX\lambda v' &\\
=&\quad (-z-xY\lambda-(1+\theta)\lambda  Y +zX\lambda)v' + Y\lambda\theta u' -Y\theta J u'+ (-\lambda z-zX+xY+(1+\theta)Y)\J v' 
\end{align*}
The last term must be zero, since $Jv'$ does not lie in $\g$. Hence, expanding $J u'$, we have
\begin{align*}
 & (-z-xY\lambda-(1+\theta)\lambda Y +zX\lambda)v' + (Y\lambda-Y\theta X)u' -(Y^2\theta)w' &
\end{align*}
and since the three vectors form a basis, each term is zero. But the $w'$-coefficient, $\theta Y^2$, cannot be zero since $\J$ would have an invariant direction, a contradiction.
\end{enumerate}

This proves (2) -- $\J$ does not have a non-trivial invariant subspace in $\g$. This in turn means that for any basis $\{u,v,w\}$ of $\g$
\begin{itemize}
\item $\{u,v,w,\J u,\J v,\J w\}$ spans a 6-dimensional space so it is a basis of $\g\times\g$.
\item $\{J^*u,J^*v,J^*w\}$ is a basis of $\g^*$.
\end{itemize}

We will now prove (3) -- a quasi-invariant vector $u$ cannot be of the form $(x,y,0)$ (or: its adjoint must have a non-zero real eigenvalue). Suppose to the contrary that
\begin{enumerate}
\item $u=e_1$ -- then $N(e_1,e_3)$ reads (for $Je_3=Xe_1+Ye_2+Ze_3$)
\begin{align*}
& [e_1,e_3]+\lambda J[e_1,e_3]+J[e_1,Je_3]-\lambda[e_1,Je_3] \\
=&\quad e_1+\lambda Je_1+ZJe_1-\lambda Ze_1 = (1+\lambda^2)e_1
\end{align*}
because the two other terms cancel out, a contradiction.
\item $u=e_2$ and compute $N(e_2,e_3)$ (with $Je_3$ as before)
\begin{align*}
&[e_2,e_3]+\lambda J[e_2,e_3]+J[e_2,Je_3]-\lambda[e_2,Je_3] \\
=&\quad \theta e_2+\lambda\theta Je_2+Z\theta Je_2-\lambda\theta Ze_2 = \theta(1+\lambda^2)e_2
\end{align*}
which gives the contradiction just as before. This two special cases are easily seen to preclude any vector $u=(x,y,0)$ from being quasi-invariant. This proves (3).
\end{enumerate}

Without loss of generality, fix a quasi-invariant vector $u=(x,y,1)$ and its Jordan basis $\{u,e_1,e_2\}$. We first compute $N(u,e_1)$ (with $Je_1=Xu+Ye_1+Ze_2$ and $Je_2=Au+Be_1+Ce_2$).
\begin{align*}
&[u,e_1]+\lambda J[u,e_1] + J[u,Xu+Ye_1+Ze_2]-\lambda [u,Xu+Ye_1+Ze_2] &\\ 
=&\quad -e_1-\lambda Je_1 -YJe_1-Z\theta Je_2 + \lambda Ye_1+\lambda Z\theta e_2 &\\
=&\quad (-1-\lambda Y-Y^2-Z\theta B+\lambda Y)e_1+(-\lambda Z-YZ-Z\theta C+\lambda Z\theta)e_2+ \\
&\qquad + (-\lambda X-XY-Z\theta A)u
\end{align*}
or 
$$
\begin{cases}
-1-Y^2-Z\theta B= 0 \\
-\lambda Z-YZ-Z\theta C+\lambda Z\theta= 0 \\
-\lambda X-XY-Z\theta A=0
\end{cases}
$$

Analogously for $N(u,e_2)$
\begin{align*}
& [u,e_2]+\lambda J[u,e_2] + J[u,Au+Be_1+Ce_2]-\lambda [u,Au+Be_1+Ce_2] & \\ 
=&\quad -\theta e_2-\lambda\theta Je_2 -BJe_1-C\theta Je_2 + \lambda Be_1+\lambda C\theta e_2 & \\
=&\quad (-\lambda\theta B-BY-C\theta B+\lambda B)e_1+(-\theta-\lambda\theta C -BZ-C^2\theta+\lambda C\theta)e_2+\\
&\qquad+(-\lambda \theta A-XB-C\theta A)u
\end{align*}
or 
$$
\begin{cases}
-\lambda\theta B-BY-C\theta B+\lambda B = 0 \\
-\theta -BZ-C^2\theta = 0 \\
-\lambda \theta A-XB-C\theta A=0
\end{cases}
$$

And finally for $N(e_1,e_2)$
\begin{align*}
& [e_1,e_2]+J[Je_1,e_2] + J[e_1,Je_2]-[Je_1,Je_2] & \\
=&\quad [e_1,e_2]+J[Xu,e_2] + J[e_1,Au]-[Xu,Be_1+Ce_2]-[Ye_1+Ze_2,Au] & \\ 
=&\quad -X\theta J e_2+AJe_1+XBe_1+XC\theta e_2 -AYe_1-AZ\theta e_2 & \\
=&\quad (-X\theta B+AY+XB-AY)e_1+(-X\theta C+AZ+X\theta C-AZ\theta)e_2+ \\
&\qquad +(-XA\theta+AX )u
\end{align*}
or 
$$
\begin{cases}
(1-\theta)XB =0\\
(1-\theta)AZ =0\\
(1-\theta)AX =0
\end{cases}
$$
in which the expressions in parentheses can be omitted because $\theta\neq 1$. Let us put these nine equations together
$$
\begin{cases}
XB=0 \\
AZ=0 \\
AX=0 \\
-\lambda\theta B-BY-C\theta B+\lambda B = 0 \\
-\theta -BZ-C^2\theta = 0 \\
-\lambda \theta A-XB-C\theta A=0 \\
-1-Y^2-Z\theta B= 0 \\
-\lambda Z-YZ-Z\theta C+\lambda Z\theta= 0 \\
-\lambda X-XY-Z\theta A=0 \\
\end{cases}
$$
Note that the fifth and seventh equation prevent $Z$ and $B$ from being zero and thus force $A$ and $X$ to equal zero, and thus simplify the picture. We are left with
$$
\begin{cases}
-\theta -C^2\theta = BZ \\
-1-Y^2=Z\theta B \\
-\lambda\theta-Y-C\theta+\lambda= 0 \\
\lambda \theta-Y-C\theta -\lambda= 0
\end{cases}
\text{{}\qquad  or\qquad {}}
\begin{cases}
-\theta -C^2\theta = BZ \\
-1-Y^2=Z\theta B \\
Y+C\theta= 0 \\
\lambda( \theta-1)= 0 
\end{cases}
$$
which means that $\lambda$ is zero.

Let us now compute the $\g^*$-part of the Nijenhuis brackets $N(u,e_1)$ and $N(u,e_2)$.
\begin{align*}
&\lambda J^*[u,e_1] + J^*[u,Ye_1+Ze_2]- [J^*u,J^*e_1] \\ 
=&\quad -\lambda J^*e_1 -YJ^*e_1-Z\theta J^*e_2 -[J^*u,J^*e_1]
\end{align*}

\begin{align*}
 & \lambda J^*[u,e_2] + J^*[u,Be_1+Ce_2]-[J^*u,J^*e_2] \\ 
=&\quad -\lambda\theta J^*e_2 -BJ^*e_1-C\theta J^*e_2 -[J^*u,J^*e_2]  
\end{align*}
A simple computation would also show that $[J^*e_1,J^*e_2]$ is zero but we don't need this. We have that in the basis $\{J^*u,J^*e_1,J^*e_2\}$ the adjoint $[J^*u,\cdot]$ is of the form
$$
\left[
\begin{array}{ccc}
 0 & 0 & 0 \\
 0 & -Y & -B \\ 
 0 & -Z\theta & Y 
\end{array}
\right]
$$
Note that we finally discarded $\lambda$'s and made use of $-Y=C\theta$. The characteristic polynomial of $[J^*u,\cdot]$ is $-t(t^2-ZB\theta-Y^2)$. But we can now plug in the last unused equation $Z\theta B=-1-Y^2$ to get $-t(t^2+1)$ with a complex root $i$, which a characteristic polynomial of an adjoint to a vector from $\g^*$ should not have. This contradiction shows that there is no integrable complex structure on the algebra of type (4) with $\theta \neq 1$.
\end{proof}

\begin{proof}[Proof of the non-existence of integrable complex structures -- case (5)]
We now turn to algebra (5) from Proposition~\ref{alg}. The proof is very similar -- we apologise if it appears indistinguishable -- but the computations must be adapted. We present them for completeness.

In the canonical basis, the adjoint automorphism $[(x,y,z),\cdot]$ is of the form
$$
\left[
\begin{array}{ccc}
 -z & -z & x+y \\
 0 & -z & y \\ 
 0 & 0 & 0 
\end{array}
\right]
$$
and has three possible types of the Jordan form:
\begin{enumerate}
\item if $z$ and $y$ are 0, then -- in the basis $\{u,v,w\}=\{xe_1,e_3,e_2\}$ -- it is
$$
\left[
\begin{array}{ccc}
 0 & 1 & 0 \\
 0 & 0 & 0 \\ 
 0 & 0 & 0 
\end{array}
\right]
$$
and the brackets are $[u,v]=u$, $[u,w]=0$ and $[v,w]=-\frac{1}{x}u-w$.
\item if $z$ is 0 and $y$ is not, then -- in the basis $\{u,v,w\}=\{(x,y,0),(x+y,y,0),(0,0,1)\}$ -- it is
$$
\left[
\begin{array}{ccc}
 0 & 0 & 0 \\
 0 & 0 & 1 \\ 
 0 & 0 & 0 
\end{array}
\right]
$$
and the brackets are $[u,v]=0$, $[u,w]=v$ and $[v,w]=2v-u$.
\item if $z$ is non-zero, then -- in the basis $\{u,v,w\}=\{(x,y,z),e_1,\frac{-1}{z}e_2 \}$ -- it is
$$
\left[
\begin{array}{ccc}
 0 & 0 & 0 \\
 0 & -z & 1 \\ 
 0 & 0 & -z 
\end{array}
\right]
$$
and the brackets are $[u,v]=-zv$, $[u,w]=v-zw$ and $[v,w]=0$.
\end{enumerate}

Suppose we have an integrable complex structure $\J$ on $\g\times\g$. Again, we proceed as follows:
{\renewcommand\labelenumi{(\theenumi)}
\begin{enumerate}
\item There is a quasi-invariant vector $u$.
\item There is no $V$, non-trivial $\J$-invariant subspace of $\g$.
\item Any quasi-invariant vector $u$ must be generic -- adjoint $[u,\cdot]$ must have a non-zero real eigenvalue.
\end{enumerate}
}

Regarding (2) -- again such a $V$ would be 2-dimensional. We consider three cases:
\begin{enumerate}
\item $[V,V]=0$. Then $V=span\{e_1,e_2\}$ and there is a quasi-invariant $u=(x,y,1)$ -- note the last coefficient. For an $a \in V$ we compute $N(u,a)$
\begin{align*}
& [u,a]+J[Ju,a]+J[u,Ja]-[Ju,Ja] \\
=&\quad [u,a]+\lambda J[u,a] + J[u,Ja]-\lambda [u,Ja]  \\
=&\quad [e_3,a]+\lambda  J[e_3,a] + J[e_3,Ja]-\lambda [e_3,Ja] 
\end{align*}
because the other terms vanish. For $a=e_1$ and $Je_1=Xe_1+Ye_2$ (which forces $Je_2=\frac{-1-X^2}{Y}e_1-Xe_2$) we get
\begin{align*}
& [e_3,e_1]+\lambda J[e_3,e_1] + J[e_3,Je_1]-\lambda [e_3,Je_1]  &\\
=&\quad -e_1-\lambda Je_1 + J[e_3,Xe_1+Ye_2]-\lambda [e_3,Xe_1+Ye_2]  &\\
=&\quad -e_1-\lambda Xe_1-\lambda Ye_2 - X(Xe_1+Ye_2)-Y(Xe_1+Ye_2)-Y\frac{-1-X^2}{Y}e_1-Xe_2+\\
&\qquad +\lambda Xe_1+\lambda Ye_1+\lambda Y e_2 
\end{align*}
or
$$
\begin{cases}
-1-\lambda X-X^2-XY+1+X^2+\lambda X+\lambda Y=\lambda Y-XY=0 \\ 
-\lambda Y-XY-Y^2+XY+\lambda Y=Y^2=0
\end{cases}
$$
which is a contradiction -- $Y$ cannot be 0, since $\J$ does not have real eigenvectors ($e_1$ would then be one).
\item $[V,V]\neq 0$, $[V,V]\subset V$. This must be a 1-dimensional subspace. A non-zero vector $v'$ in this subspace is a non-zero eigenvector for any linearly independent vector $u'$ in $V$ -- this $u'$ can be assumed to be $(x,y,1)$ and since $v'$ must be proportional to $e_1$, it may be assumed to be precisely $e_1$. The third vector $w'=e_2$ is not in $V$ and completes a Jordan basis for the adjoint of $u'$. Write the quasi-invariant vector $u=xu'+ye_1+e_2$ (note the last coefficient) in that basis. We compute the full Nijenhuis bracket $N(u,e_1)$ with $\J e_1=Je_1=Xu'+Ye_1$.
\begin{align*}
& [u,e_1]+\lambda \J[u,e_1] + \J[u,Xu'+Ye_1]-\lambda [u,Xu'+Ye_1]  \\
=&\quad [xu',e_1]+\lambda \J[xu',e_1] + \J[xu'+ye_1+e_2,Xu'+Ye_1]-\lambda [xu'+ye_1+e_2,Xu'+Ye_1]  \\
=&\quad -x e_1-\lambda x J e_1+yX J e_1-xY J e_1+X J e_1+X \J e_2-\lambda yX e_1+\lambda xY e_1-\lambda Xe_1-\lambda X e_2  
\end{align*}
This, however, means that $X \J e_2$ is in $\g$, which is possible only if $X$ is zero, or $\g$ would be invariant under $J$. But then $\J$ has an invariant direction $e_1$, a contradiction.
\item $[V,V]\neq 0$, $[V,V]$ is not in $V$. $[V,V]$ is contained in $[\g,\g]=span\{e_1,e_2\}$ which intersects $V$ along a line $\mathbb{R} \cdot u'$. The adjoint of this $u'$ has a Jordan basis $\{u',v',w'\}$ for $w'\in V$ rescaled to be of the form $(a,b,1)$ and $v'=[u',w'] \in [V,V]$. We have that $[v',w']=2v'-u'$. The quasi-invariant vector $u=xu'+v'+zw'$ (note the middle coefficient) and $u'$ with $\J u'=Ju'=Xu'+Yw'$ give the full $N(u,u')$
\begin{align*}
& [u,u']+\lambda \J[u,u'] + \J[u,\J u']-\lambda [u,\J u'] & \\
=&\quad [xu'+v'+zw',u']+\lambda \J[xu'+v'+zw',u']+ \\
&\qquad +\J[xu'+v'+zw',Xu'+Yw']-\lambda[xu'+v'+zw',Xu'+Yw'] & \\
=&\quad [zw',u']+\lambda \J[zw',u']+\J[zw',Xu']+\J[xu'+v',Yw']-\lambda[xu'+v',Yw']-\lambda[zw',Xu']  &\\
=&\quad -zv'-\lambda z\J v'-zX\J v'+xY\J v'+2Y\J v'-Y\J u'-xY\lambda v'-\lambda  Y2v'+\lambda  Yu'+zX\lambda v' &\\
=&\quad (-z-xY\lambda-2\lambda  Y +zX\lambda)v' + (Y\lambda)u' -Y J u'+ (-\lambda z-zX+xY) \J v' 
\end{align*}
The last term must be zero since $\J v'$ does not lie in $\g$. Hence we have
\begin{align*}
& (-z-xY\lambda-2\lambda Y +zX\lambda ) v' + (Y\lambda) u' -Y(Xu'+Yw')  \\
=&\quad (-z-xY\lambda-2\lambda  Y +zX\lambda) v' + (Y\lambda-YX) u' +Y^2 w' 
\end{align*}
and so each term is zero. But again, the $w'$-part cannot be zero, since $\J$ does not have an invariant direction, a contradiction.
\end{enumerate}

This proves that $\J$ does not have a non-trivial invariant space in $\g$. Again, we note that for any basis $\{u,v,w\}$ in $\g$
\begin{itemize}
\item $\{u,v,w,\J u,\J v,\J w\}$ spans a 6-dimensional space and so it is a basis of $\g\times\g$.
\item $\{J^*u,J^*v,J^*w\}$ is a basis of $\g^*$.
\end{itemize}

We will now prove the last point -- a quasi-invariant vector $u$ cannot be of the form $(x,y,0)$. Suppose to the contrary that it can. 
\begin{enumerate}
\item $u=e_1$ -- then $N(e_1,e_3)$ reads (for $Je_3=Xe_1+Ye_2+Ze_3$)
\begin{align*}
& [e_1,e_3]+\lambda J[e_1,e_3]+J[e_1,Je_3]-\lambda[e_1,Je_3] \\
=&\quad e_1+\lambda Je_1+ZJe_1-\lambda Ze_1 = (1+\lambda^2)e_1
\end{align*}
\item $u=(x,y,0)$ for a non-zero $y$ with a Jordan basis $\{u,v,e_3\}$ -- again compute (with $Jv=Xu+Yv+Ze_3$) $N(u,v)$
\begin{align*}
& [u,v]+\lambda J[u,v] + J[u,Jv]-\lambda [u,Jv]  \\
=&\quad J[u,Ze_3]-\lambda [u,Ze_3] \\
=&\quad ZJv-\lambda Z v=ZXu+ZYv+Z^2e_3-\lambda Zv
\end{align*}
which gives $Z=0$. Now for $Je_3=Au+Bv+Cw$ compute $N(u,e_3)$
\begin{align*}
& [u,e_3]+\lambda J[u,e_3] + J[u,Je_3]-\lambda [u,Je_3]  \\
=&\quad v+\lambda Jv + J[u,Ce_3]-\lambda [u,Ce_3] \\
=&\quad v+\lambda Jv + CJv-\lambda C v \\
=&\quad (1+\lambda Y+CY-\lambda C)v+(\lambda X+CX)u
\end{align*}
Note we suppressed the vanishing $Ze_3$ term. This means $X(\lambda+C)=0$. However, if $\lambda+C$ equals zero  then the first coefficient reads $1+\lambda^2$, a contradiction. Hence we got $X=0$. We continue to suppress the vanishing terms in the last remaining bracket $N(v,e_3)$.
\begin{align*}
& [v,e_3]+J[Jv,e_3] + J[v,Je_3]-[Jv,Je_3]  \\
=&\quad 2v-u+2YJv-YJu+J[v,Ce_3]-[Yv,Ce_3]  \\
=&\quad 2v-u+2Y^2v-\lambda Yu+2CYv-\lambda Cu-2YCv+CYu  \\
=&\quad (2+Y^2+2CY-2CY)v+(-1-\lambda Y-\lambda C+CY)u
\end{align*}
which is of course a contradiction on $2+Y^2=0$.
\end{enumerate}

This proves that again only the generic vectors can be quasi-invariant.

For such a vector $u=(x,y,1)$ and its Jordan basis $\{u,v,w\}$ we compute $N(u,v)$ with $Jv=Xu+Yv+Zw$ and $Jw=Au+Bv+Cw$
\begin{align*}
& [u,v]+\lambda J[u,v] + J[u,Xu+Yv+Zw]-\lambda [u,Xu+Yv+Zw] & \\ 
=&\quad -v-\lambda  Jv -YJv+ZJv-ZJw+\lambda Yv-\lambda Zv+Z\lambda w & \\
=&\quad (-1+\lambda Y-\lambda Z)v+(Z\lambda)w+(-\lambda -Y+Z)Jv-ZJw & \\
=&\quad (-1+\lambda Y-\lambda Z+Y(-\lambda -Y+Z)-ZB)v+(Z\lambda+Z(-\lambda -Y+Z)-ZC)w + \\
&\qquad + (X(-\lambda -Y+Z)-ZA) u
\end{align*}
or 
$$
\begin{cases}
-1+\lambda Y-\lambda Z+Y(-\lambda -Y+Z)-ZB = 0 \\
Z\lambda+Z(-\lambda -Y+Z)-ZC = 0 \\
X(-\lambda -Y+Z)-ZA=0
\end{cases}
$$

Now compute $N(v,w)$
\begin{align*}
& [v,w]+J[Jv,w] + J[v,Jw]-[Jv,Jw] & \\
=&\quad 0+J[Xu,w]+J[v,Au]-[Xu+Yv+Zw,Au+Bv+Cw] & \\
=&\quad XJv-XJw+AJv-X[u,Bv+Cw]+A[u,Yv+Zw] & \\
=&\quad XJv-XJw+AJv+XBv-XCv+XCw-AYv+AZv-AZw & \\
=&\quad (XB-XC-AY+AZ)v+(XC-AZ)w+(X+A)Jv-XJw & \\
=&\quad (XB-XC-AY+AZ+XY+AY-XB)v + \\
&\qquad +(XC-AZ+XZ+AZ-XC)w+(X^2+AX-XA)u
\end{align*}
or 
$$
\begin{cases}
XC+AZ+XY = 0 \\
XZ = 0 \\
X^2=0 \\
\end{cases}
$$
which gives $X=0$ and subsequently $AZ=0$. If $Z$ was 0, then the $v$-coefficient in the previous equation would read $-1-Y^2$, which cannot be zero. Before we jump to any further conclusions, compute $N(u,w)$
\begin{align*}
& [u,w]+\lambda J[u,w] + J[u,Jw]-\lambda[u,Jw] & \\
=&\quad v-w+\lambda Jv-\lambda  Jw+J[u,Bv+Cw]-\lambda[u,Bv+Cw] & \\
=&\quad v-w+\lambda Xu + \lambda Yv+\lambda Zw -\lambda  Au-\lambda  Bv-\lambda  Cw- B(Xu+Yv+Zw) +\\
&\qquad +C(Xu+Yv+Zw)-C(Au+Bv+Cw)+\lambda Bv-\lambda Cv+\lambda  Cw 
\end{align*}
or 
$$
\begin{cases}
1+\lambda Y-\lambda B-BY+CY-CB+\lambda B-\lambda C = 0 \\
-1+\lambda Z-\lambda  C-BZ+CZ-C^2+\lambda  C =0 \\
\lambda X -\lambda A-BX+CX-CA=0 
\end{cases}
$$
Putting these equations together we get
$$
\begin{cases}
X = 0 \\
A = 0 \\
-1-\lambda Z-Y^2+ZY-ZB = 0 \\
1+\lambda Y- BY+CY-BC -\lambda C = 0 \\
-YZ+Z^2-ZC = 0 \\
-1+\lambda Z -BZ+CZ-C^2 = 0
\end{cases}
$$
The last two give
\begin{align*}
C&=-Y+Z\\
B&=\frac{-1+\lambda Z +CZ-C^2}{Z}
\end{align*}
that combine into
$$
B=\frac{\lambda Z+YZ-Y^2-1}{Z}
$$

We now compute the $\g^*$-parts of the Nijenhuis brackets to examine $[J^*u,\cdot]$ in the basis $\{J^*u,J^*v,J^*w\}$. The bracket $N(u,v)$ gives
\begin{align*}
& \lambda J^*[u,v]+J^*[u,Yv+Zw]-[J^*u,J^* v]  \\
=&\quad -\lambda J^*v-YJ^*v+ZJ^*v-ZJ^*w-[J^*u,J^* v] 
\end{align*}

Similarly for $N(u,w)$
\begin{align*}
& \lambda J^*[u,w]+J^*[u,Bv+Cw]-[J^*u,J^* w]  \\
=&\quad \lambda J^*v-\lambda J^*w-BJ^*v+CJ^*v-CJ^*w-[J^*u,J^* w] 
\end{align*}

It is again easily checked that the last Nijenhuis bracket gives $[J^*v,J^* w]=0$ (which we again don't need). This presents the adjoint $[J^*u,\cdot]$ in the basis $\{J^*u,J^*v,J^*w\}$ in the following form
$$
\left[\begin{array}{ccc}
0 & 0 & 0 \\
0 & -\lambda-Y+Z & \lambda - B + C  \\
0 & -Z & -\lambda-C
\end{array}\right]
$$
This matrix must have a non-zero real eigenvalue since $J^*u$ is quasi-invariant. This eigenvalue must be double. We compute the relevant part of the characteristic polynomial
\begin{align*}
& (-\lambda-Y+Z -t)(-\lambda-C-t)+Z(\lambda - B + C) \\
=&\quad t^2+t(\lambda+Y-Z+\lambda+C)+Z(\lambda - B+C)+\lambda^2+\lambda C+\lambda Y+YC-Z\lambda-CZ \\
=&\quad t^2 + t(2\lambda+Y+C-Z)-BZ+\lambda^2+\lambda C+\lambda Y+YC 
\end{align*}
and its discriminant 
\begin{align*}
& (2\lambda+Y+C-Z)^2-4(-BZ+\lambda^2+\lambda C+\lambda Y+YC) & \\
=&\quad 4\lambda^2+Y^2+C^2+Z^2+4\lambda Y+4\lambda C-4\lambda Z + \\
&\qquad +2YC-2YZ-2ZC+4BZ-4\lambda^2-4\lambda C-4\lambda Y-4YC & \\
=&\quad Y^2+C^2+Z^2-4\lambda Z -2YC-2YZ-2ZC+4BZ  
\end{align*}
which must be zero to yield a unique root. Observe that $B$ appears only once and multiplied by a non-zero factor $Z$. We can therefore write
$$
B=-\frac{1}{4Z}\left( Y^2+C^2+Z^2-4\lambda Z -2YC-2YZ-2ZC\right)
$$
and substitute $C=-Y+Z$ to get
$$
B=\frac{\lambda Z+YZ-Y^2}{Z}
$$

The two expressions we obtained for $B$ differ by a non-zero element $\frac{1}{Z}$, giving a contradiction. Hence there are no integrable complex structures on this algebra as well.
\end{proof}

We will now present a concrete example of an integrable complex structure in each possible case. Recall how we defined the complex structures in the proof of Theorem~\ref{main}: $\J u = u^*$, $\J v=w$, $\J v^* = w^*$. To exhibit such a structure we only need to write down the appropriate $u$, $v$, and $w$.

\begin{pro}\label{STR}
The following vectors satisfy the conditions given in the proof of Theorem~\ref{main} (and thus give integrable complex structures as above) for the corresponding algebras of Proposition~\ref{alg}
{\renewcommand\labelenumi{(\theenumi)}
\begin{enumerate}
\item $u=e_1$, $v=e_2$, $w=e_3$.
\item $u=e_3$, $v=e_1$, $w=e_2$.
\item $u=e_3$, $v=e_1$, $w=e_2$.
\item for $\theta=1$, $u=e_3$, $v=e_1$, $w=e_2$.
\setcounter{enumi}{5}
\item $u=e_3$, $v=e_1$, $w=e_2$.
\item $u=e_2$, $v=e_1$, $w=e_3$.
\item $u=e_3$, $v=e_1$, $w=e_2$.
\end{enumerate}
}
\end{pro}

\section{Integrable complex structures on $\oen\times\oen$}

To finish with a foray into higher dimensions, we give a concrete example of an integrable complex structure on $\oen\times\oen$. While this example can be recovered from \cite{1}, it is a by-product of the above considerations, and our explicit form proves, perhaps, useful for future geometric applications.

Recall that $\oen$ is generated by the elementary antisymmetric matrices
$$
e_{ij}=\big[ \delta_i^j\big]-\big[ \delta_j^i\big]
$$
where $\delta$ is the Kronecker symbol. Each matrix has a single 1 in its $i$-th row and $j$-th column, and $-1$ in the opposite entry. This is only a generating set, we will be using the basis $\{e_{ij}\}_{i<j}$ shortly. Recall that the bracket in $\oen$ is given by
$$
[e_{ij},e_{jk}]=e_{ik}
$$

Using our previous notation to distinguish the two copies of $\oen$, we define the following complex structure~$\J$~as follows:
$$
\begin{cases}
\J e_{12} = e_{12}^* \\
\J e_{1i} = e_{2i} & \text{for $i>2$} \\
\J e_{1i}^* = e_{2i}^* & \text{for $i>2$} \\
\J e_{34} = e_{34}^* \\
\J e_{3i} = e_{4i} & \text{for $i>4$} \\
\J e_{3i}^* = e_{4i}^* & \text{for $i>4$} \\
\vdots \\
\J e_{(n-1)n}= e_{(n-1)n}^* & \text{or}\\
\begin{cases}
\J e_{(n-2)n}= e_{(n-1)n} \\
\J e_{(n-2)n}^*= e_{(n-1)n}^* 
\end{cases}
\end{cases}
$$
where the slight discrepancy comes from parity of $n$ -- but is without any impact on the construction. Note we chose a family of quasi-invariant vectors (such as $e_{12}$) and assign to them subspaces of their eigenspaces. Note however that, for example, $e_{14}$ is a (complex) eigenvector for both $e_{12}$ and $e_{34}$, but we assign it to the former. 

\begin{thm}
For every $n\geq 2$, the above complex structure is integrable.
\end{thm}

\begin{proof}
We need to compute the Nijenhuis bracket only in the four following cases. Assume $i<j<k<l$.
\begin{enumerate}
\item Two different quasi-invariant vectors. But in
$$
[e_{ij},e_{kl}]+\J[\J e_{ij},e_{kl}]+\J[e_{ij},\J e_{kl}]-[\J e_{ij},\J e_{kl}]
$$
every term vanishes by the choice of indices.
\item Two vectors assigned to the same quasi-invariant vector. Then
\begin{align*}
&[e_{ij},e_{ik}]+\J[\J e_{ij},e_{ik}]+\J[e_{ij},\J e_{ik}]-[\J e_{ij},\J e_{ik}] & \\
=\quad&[e_{ij},e_{ik}]+\J[e_{(i+1)j},e_{ik}]+\J[e_{ij},e_{(i+1)k}]-[e_{(i+1)j},e_{(i+1)k}] & \\
=\quad&-e_{jk}+e_{jk}&
\end{align*}
because the brackets under $\J$ vanished (the assigned vectors as written cannot bear indices differing by 1).
\item Two vectors assigned to different quasi-invariant vectors. The only non-trivial situation is 
\begin{align*}
&[e_{ij},e_{jk}]+\J[\J e_{ij},e_{jk}]+\J[e_{ij},\J e_{jk}]-[\J e_{ij},\J e_{jk}] & \\
=\quad&[e_{ij},e_{jk}]+\J[e_{(i+1)j},e_{jk}]+\J[e_{ij},e_{(j+1)k}]-[e_{(i+1)j},e_{(j+1)k}] & \\
=\quad&e_{ik}+\J e_{(i+1)k} = e_{ik}-e_{ik}&
\end{align*}
\item Quasi-invariant vector and a vector assigned to a quasi-invariant vector of lower indices (the reverse situation is trivial, because none of the indices can match). Note the index alternation in the last step that comes from the inequality between $i$ and $j$.
\begin{align*}
&[e_{ij},e_{j(j+1)}]+\J[\J e_{ij},e_{j(j+1)}]+\J[e_{ij},\J e_{j(j+1)}]-[\J e_{ij},\J e_{j(j+1)}] & \\
=\quad&[e_{ij},e_{j(j+1)}]+\J[e_{(i+1)j},e_{j(j+1)}]+\J[e_{ij},e_{j(j+1)}^*]-[e_{(i+1)j},e_{j(j+1)}^*] & \\
=\quad&e_{i(j+1)}+\J e_{(i+1)(j+1)} = e_{i(j+1)}-e_{i(j+1)}&
\end{align*}
\item Every other case is either symmetric (concerns the $\oen^*$ counterparts), redundant (by Proposition~\ref{redu}) or trivial (concerns a 2-dimensional invariant space). 
\end{enumerate}
\end{proof}

\end{document}